\documentclass[10pt]{amsart}

\makeatletter
\def\blfootnote{\gdef\@thefnmark{}\@footnotetext}
\makeatother

\usepackage{epsfig}
\usepackage{graphics}
\usepackage{dcpic, pictexwd}

\usepackage[colorlinks=true,
                    linkcolor=blue,
                    urlcolor=blue,
                    citecolor=blue,
                    anchorcolor=blue]{hyperref}
\usepackage{mathtools}
\usepackage{amsmath,amssymb}
\theoremstyle{plain}

\newtheorem*{theorem*}{Theorem}
\newtheorem*{thma}{Theorem A}
\newtheorem*{thmb}{Theorem B}

\newtheorem{theorem}{Theorem}[section]

\newtheorem{lemma}[theorem]{Lemma}
\newtheorem{proposition}[theorem]{Proposition}

\theoremstyle{remark}

\theoremstyle{Acknowledgments}

\newtheorem*{cora}{Corollary A}
\newtheorem*{corb}{Corollary B}

\theoremstyle{definition}

\def\mod{{\rm Mod}}

\begin{document}
\blfootnote{\textup{2000} \textit{Mathematics Subject Classification}:
57N05, 20F38, 20F05}
\blfootnote{\textit{Keywords}:
Mapping class groups, nonorientable surfaces, twist subgroup, torsion, generating sets, commutators}
\newenvironment{prooff}{\medskip \par \noindent {\it Proof}\ }{\hfill
$\square$ \medskip \par}
    \def\sqr#1#2{{\vcenter{\hrule height.#2pt
        \hbox{\vrule width.#2pt height#1pt \kern#1pt
            \vrule width.#2pt}\hrule height.#2pt}}}
    \def\square{\mathchoice\sqr67\sqr67\sqr{2.1}6\sqr{1.5}6}
\def\pf#1{\medskip \par \noindent {\it #1.}\ }
\def\endpf{\hfill $\square$ \medskip \par}
\def\demo#1{\medskip \par \noindent {\it #1.}\ }
\def\enddemo{\medskip \par}
\def\qed{~\hfill$\square$}

 \title[ The Twist Subgroup is generated by two elements] {The Twist Subgroup is generated by two elements}

\author[T{\"{u}}l\.{i}n Altun{\"{o}}z,       Mehmetc\.{i}k Pamuk, and O\u{g}uz Y{\i}ld{\i}z ]{T{\"{u}}l\.{i}n Altun{\"{o}}z,    Mehmetc\.{i}k Pamuk, and O\u{g}uz Y{\i}ld{\i}z}

\address{Department of Mathematics, Middle East Technical University,
 Ankara, Turkey}
\email{atulin@metu.edu.tr}  \email{mpamuk@metu.edu.tr} \email{oguzyildiz16@gmail.com}


\begin{abstract}
We show that the twist subgroup $\mathcal{T}_g$ of a nonorientable surface of genus $g$ can be generated by two elements for every odd $g\geq27$ and even $g\geq42$. Using these generators, we can also show that  $\mathcal{T}_g$ can be generated by two or three commutators depending on $g$ modulo $4$. Moreover, we show that $\mathcal{T}_g$ can be generated by three elements if $g\geq 8$. For this general case, the number of commutator generators is either three or four depending on $g$ modulo $4$ again.
\end{abstract}
\maketitle
  \setcounter{secnumdepth}{2}
 \setcounter{section}{0}
\section{Introduction}

Let $N_g$ denote a closed connected nonorientable surface of genus $g$. The isotopy classes of self-diffeomorphisms of $N_g$ form a group, $\mod(N_g)$, called 
the mapping class group of  $N_g$. Compared to the corresponding group for the orientable surfaces, this group is much less studied. 
It is known that~\cite{l1,l2} $\mod(N_g)$  is generated by Dehn twists about two-sided simple closed curves and a so-called $Y$-homeomorphism (or a crosscap slide). 
Dehn twists about two-sided simple closed curves form an index $2$ subgroup $\mathcal{T}_g$ of $\mod(N_g)$, called the twist subgroup.

The study of algebraic properties of mapping class groups is an active one leading to interesting developments.
In this paper, we study the problem of finding generating sets for $\mathcal{T}_g$ with minimal number of elements. Previously, Du~\cite{du} for $g\geq5$ and odd, 
and the authors~\cite{apy1} for $g\geq13$ obtained generating sets for $\mathcal{T}_g$ with three elements.  Recently, Lesniak and Szepietowski \cite{ls} showed that for 
$g\neq 4$, $\mod(N_g)$ and $\mathcal{T}_g$ can be generated by three torsion elements.
Since the twist subgroup is not cyclic, any generating 
set must contain at least two elements. In this direction, modifying the techniques of Baykur and Korkmaz \cite{bk} for nonorientable surfaces, we obtain the following optimal generating 
set consisting of a finite order ($g$ or $g-1$) and an infinite order mapping classes (see Theorems~\ref{t29} and \ref{t42}):

\begin{thma}\label{t0}
The twist subgroup $\mathcal{T}_g$ of $\mod(N_g)$ is generated by two elements for every odd $g\geq27$ and even $g\geq42$.
\end{thma}

The twist subgroup $\mathcal{T}_{g}$ admits an epimorphism to the automorphism group of $H_1(N_g;\mathbb{Z}_2)$ 
preserving the modulo $2$ intersection pairing~\cite{mpin}, which is isomorphic to (see~\cite{sz4} pp.338--339 )
\begin{eqnarray}
 \begin{cases} Sp(2h;\mathbb{Z}_2)&\text{if $g=2h+1$,} \\ 
 Sp(2h;\mathbb{Z}_2)\ltimes \mathbb{Z}_{2}^{2h+1}&\text{if $g=2h+2$.}  \end{cases}
\nonumber 
\end{eqnarray}
Hence, the action of mapping classes on $H_1(N_g;\mathbb{Z}_2)$ induces an epimorphism from $\mathcal{T}_g$ to 
the above groups, which immediately implies the following corollary:

\begin{cora}
For every odd $g\geq27$, the symplectic group $Sp\big( g-1; \mathbb{Z}_2\big)$ and for every even $g\geq42$, the 
symplectic group $Sp\big(g-2; \mathbb{Z}_2\big) \ltimes \mathbb{Z}_{2}^{g-1}$ is generated by two elements.
\end{cora}
If one wants to decrease $g$, then the number of generators increases. In this direction, we have shown that for $g\geq 9$, the twist subgroup $\mathcal{T}_g$ is generated by three elements (see Theorems~\ref{t9odd} and~\ref{t8even}). Similar results can also be obtained for the corresponding symplectic groups.

The twist subgroup $\mathcal{T}_g$ is perfect if $g\geq7$~\cite{mk3,mk4}.  As also noted in \cite{bk}, it is interesting to know for which perfect groups 
the minimal number of generators is equal to the minimal number of commutator generators. For a closed connected orientable surface of genus $g$, 
which is a perfect group for $g\geq 3$, Baykur and Korkmaz \cite{bk}, showed that the mapping class group can be generated by two  commutators 
if $g\geq 5$, and by three commutators if $g\geq 3$.  

In the nonorientable case for the twist subgroup,  we obtain the following result (see Theorem~\ref{tcom}):
\begin{thmb}\label{t}
The twist subgroup $\mathcal{T}_g$ of $\mod(N_g)$
is generated by
\begin{enumerate}
\item[(1)] two commutators if $g=4k\geq44$ or $g=4k+1\geq29$ and
\item[(2)] three commutators if $g=4k+2\geq30$ or $g=4k+3\geq43$.
\end{enumerate}
\end{thmb}

Also for this case, we have the following corollary for symplectic groups: 
\begin{corb} The symplectic group $Sp\big(g-2; \mathbb{Z}_2\big) \ltimes \mathbb{Z}_{2}^{g-1}$ is generated by 
 \begin{itemize}
\item[(1)] two commutators if $g=4k\geq44$ and 
\item[(2)] three commutators if $g=4k+2\geq30$.
 \end{itemize}
 The symplectic group $Sp\big(g-1; \mathbb{Z}_2\big)$ is generated by
 \begin{itemize}
\item[(1)] two commutators if $g=4k+1\geq29$ and 
\item[(2)] three commutators if $g=4k+3\geq43$.
 \end{itemize}
\end{corb}

If one drops $g$ to $7$, then the number of commutator generators are $1$ more than the corresponding cases in Theorem B (see Theorem~\ref{tcom1}). Moreover, one can obtain similar results for the corresponding symplectic groups.

Before we finish this section, we want to point out that throughout the paper we look for generators for $\mathcal{T}_g$ which can be expressed as a commutator.

\medskip

\noindent
{\bf Acknowledgements.}
 The authors thank \.{I}nan\c{c} Baykur for his valuable comments on an earlier version of this paper. The first author was partially supported by the Scientific and 
Technological Research Council of Turkey (T\"{U}B\.{I}TAK)[grant number 120F118].


\par  
\section{Background and Results on Mapping Class Groups} \label{S2}
 Let $N_g$ denote a closed connected nonorientable surface of genus $g$. The {\textit{mapping class group}} 
 $\mod(N_g)$ of the surface $N_g$ is defined to be the group of the isotopy classes of 
 all self-diffeomorphisms of $N_g$. Throughout the paper we consider
 diffeomorphisms up to isotopy. For the composition of two diffeomorphisms, we
use the functional notation; if $g$ and $h$ are two diffeomorphisms, 
the composition $gh$ means that $h$ is applied first.\\
\indent
A simple closed curve on $N_g$ is called
\textit{one-sided} if a regular neighbourhood of it is homeomorphic to 
a M\"{o}bius band. It is called \textit{two-sided} if a regular neighbourhood of 
it is homeomorphic to an annulus. If $a$ is a two-sided simple closed 
curve on $N_g$, to define the Dehn twist $t_a$, we need to fix one of two possible 
orientations on a regular neighbourhood of $a$ (as we did for the 
curves in Figure~\ref{G}). We will denote the right-handed Dehn twist $t_a$ about $a$ by the corresponding capital 
letter $A$.

Now, let us recall the following basic properties of Dehn twists which we use frequently in the remaining of the paper. Let $a$ and $b$ be 
two-sided simple closed curves on $N_g$ and $f\in \mod(N_g)$.
\begin{itemize}
\item \textbf{Commutativity:} If $a$ and $b$ are disjoint, then $AB=BA$.
\item \textbf{Conjugation:} If $f(a)=b$, then $fAf^{-1}=B^{s}$, where $s=\pm 1$ 
depending on whether $f$ is orientation preserving or orientation reversing on a 
neighbourhood of $a$ with respect to the chosen orientation.
\end{itemize}

We use conjugation property repeatedly throughout the paper. To avoid too much repetition we want to remind the reader the following simple algebraic property:
Let $G$ be a subset of $\mod(N_g)$ and $f \in G$. If $f$ takes the curve $a$ to $b$, that is $f(a)=b$ and the Dehn twist $A$ belongs to $G$, then the Dehn twist $B$ also belongs to $G$.

\begin{figure}[h]
\begin{center}
\scalebox{0.35}{\includegraphics{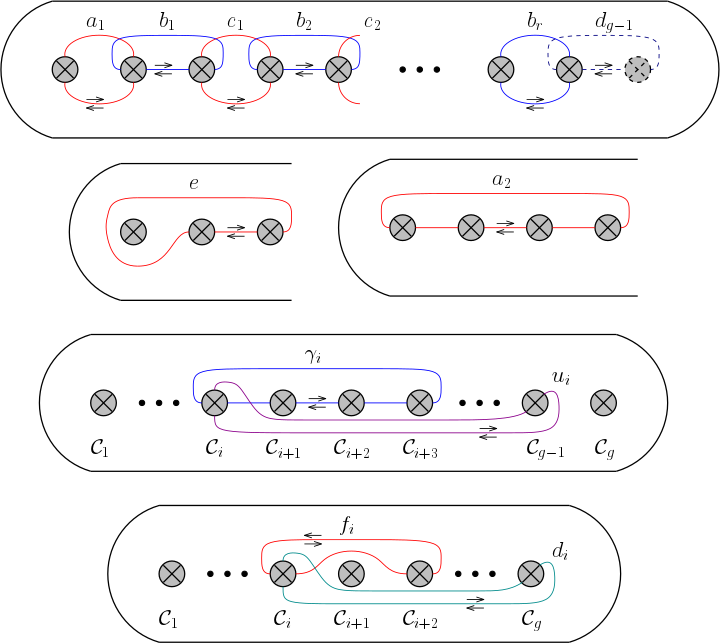}}
\caption{The curves $a_1,a_2,b_i,c_i,e,\gamma_i,u_i,f_i$ and $d_i$ on the surface $N_g$.}
\label{G}
\end{center}
\end{figure}
\par

\begin{figure}[h]
\begin{center}
\scalebox{0.14}{\includegraphics{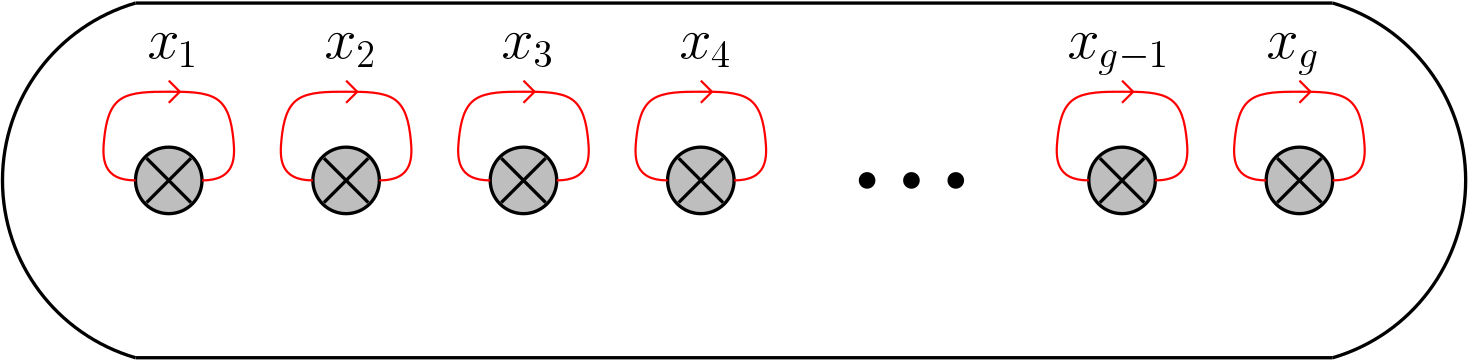}}
\caption{Generators of $H_1(N_g;\mathbb{R})$.}
\label{H}
\end{center}
\end{figure}

Consider the surface $N_g$ as in Figure~\ref{G}.
The Dehn twist generators given by Omori are as follows (note that we do not have the curve $d_{g-1}$ when $g$ is odd).
\begin{theorem}\cite{om}\label{thm1}
The twist subgroup $\mathcal{T}_g$ is generated by the following $(g+1)$ Dehn twists
\begin{enumerate}
\item $A_1,A_2,B_1,\ldots, B_r$, $C_1,\ldots, C_{r-1}$ and $E$ if $g=2r+1$ and 
\item $A_1,A_2,B_1,\ldots, B_r$, $C_1,\ldots, C_{r-1}$, $D_{g-1}$ and $E$ if $g=2r+2$.
\end{enumerate}
\end{theorem}

\noindent
Consider a basis $\lbrace x_1, x_2. \ldots, x_{g-1}\rbrace$ for $H_1(N_g; \mathbb{R})$ 
such that the curves $x_i$ are one-sided and disjoint as in Figure~\ref{H}. It is known that every
 diffeomorphism $f: N_g \to N_g$ induces a linear map 
 $f_{\ast}: H_1(N_g;\mathbb{R}) \to H_1(N_g;\mathbb{R})$. Therefore, one can
  define a homomorphism $D: \mod(N_g) \to \mathbb{Z}_{2}$ by $D(f)=\textrm{det}(f_{\ast})$. 
  The following lemma from~\cite{l1} tells when a mapping class falls into the twist subgroup $\mathcal{T}_g$.

\begin{lemma}\label{lem1} Let $f\in  \mod(N_g)$. Then  $D(f)=1$ if $f\in \mathcal{T}_g$ and
$D(f)=-1$ if $f \not \in \mathcal{T}_g$.
\end{lemma}

Before we finish this section, let us recall a generating set for $\mathcal{T}_g$ given in~\cite{apy1}. The diffeomorphism $T$ is the rotation by $\frac{2\pi}{g}$
or $\frac{2\pi}{g-1}$ for $g$ is odd or even as shown in Figures~\ref{TO} and ~\ref{TE}, respectively. 

\begin{theorem}\label{t1}
The twist subgroup $\mathcal{T}_g$ is generated by the elements 
\begin{enumerate}
\item $T,A_1A_{2}^{-1},B_1B_{2}^{-1}$ and $E$ if $g=2r+1$ and $r\geq 3$,
\item $T,A_1A_{2}^{-1},B_1B_{2}^{-1}, D_{g-1}$ and $E$ if $g=2r+2$ and $r\geq 3$.
\end{enumerate}
 \end{theorem}

\section{A generating set for $\mathcal{T}_g$}\label{S3}
In the first part of this section, where the genus of the surface is odd, we refer to Figure~\ref{TO}. Note that $T(b_i)=c_i$, $T(c_j)=b_{j+1}$ for $i,j=1,\ldots,r-1$, and $T^{2}(b_r)=a_1$ and $T(a_1)=(b_1)$. Note that the rotation $T$ satisfies $D(T)=1$, which implies that $T$ belongs to $\mathcal{T}_g$. Also, let $\Gamma_{i}$ denote the right handed Dehn twist about the curve $\gamma_i$ shown in Figure~\ref{G}.
\begin{figure}[h]
\begin{center}
\scalebox{0.22}{\includegraphics{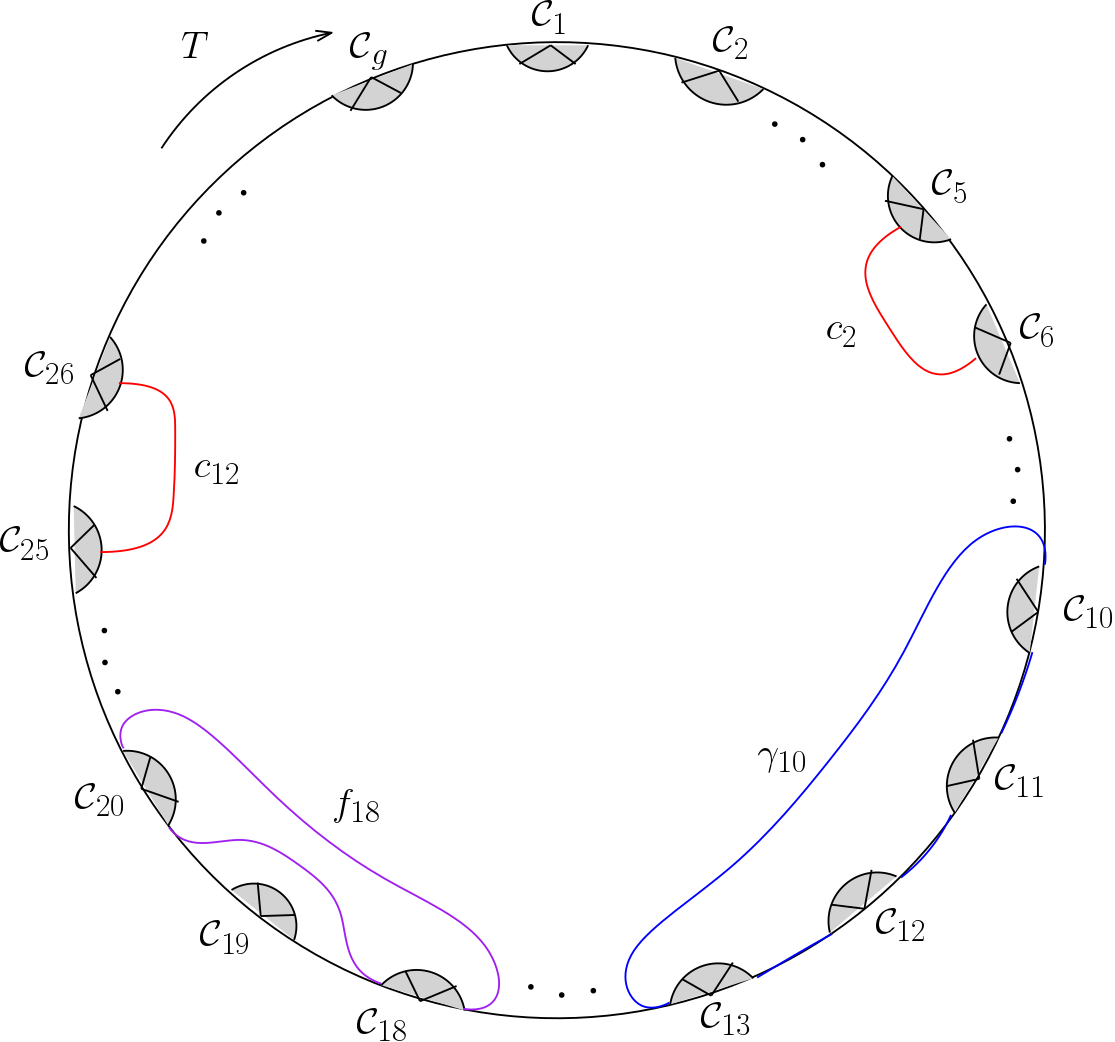}}
\caption{The rotation $T$ of $N_g$ if $g=2r+1$.}
\label{TO}
\end{center}
\end{figure}
\begin{theorem}\label{t29}
For $g=2r+1\geq27$, the twist subgroup $\mathcal{T}_g$ is generated by the elements $T$ and $\Gamma_{10}C_2^{-1}F_{18}C_{12}^{-1}$.
\end{theorem}
\begin{proof}
Consider the surface $N_g$ as in Figure~\ref{TO}.
Let $G_1:=\Gamma_{10}C_2^{-1}F_{18}C_{12}^{-1}$ and let $G$ be the subgroup of $\mathcal{T}_g$ generated by $T$ and $G_1$. It follows from Theorem~\ref{t1} that it is enough to prove that the elements $A_1A_{2}^{-1}$, $B_1B_{2}^{-1}$ and $E$ are contained in the subgroup $G$. Now, since
\[
T^{-4}(\gamma_{10},c_2,f_{18},c_{12})=(\gamma_6,a_1,f_{14},c_{10}),
\]
we have $T^{-4}G_1T^{4}=\Gamma_6A_{1}^{-1}F_{14}C_{10}^{-1}\in G$. Let $G_2:=\Gamma_6A_{1}^{-1}F_{14}C_{10}^{-1}$ and we get
\[
G_3:=(G_2G_1^{-1})G_2(G_2G_1^{-1})^{-1}=C_2A_{1}^{-1}F_{14}C_{10}^{-1} \in G.
\]
Before we proceed any further as we have similar cases in the remaining parts of the paper, let us explain some details of this calculation. It is easy to verify that the diffeomorphism $G_2G_1$ maps the curves $\gamma_6,a_1,f_{14},c_{10}$ to $c_2,a_1,f_{14},c_{10}$, respectively. Then, we have
\begin{eqnarray*}
G_3:&=&(G_2G_1^{-1})G_2(G_2G_1^{-1})^{-1}\\
       &=&(G_2G_1^{-1})\Gamma_6A_{1}^{-1}F_{14}C_{10}^{-1}(G_2G_1^{-1})^{-1}\\
       &=&C_2A_{1}^{-1}F_{14}C_{10}^{-1}.
\end{eqnarray*}
From this, we get 
\[
G_2G_{3}^{-1}=\Gamma_6C_2^{-1} \in G,
\]
and also
\[
T^{4}\Gamma_6C_2^{-1}T^{-4}=\Gamma_{10}C_4^{-1} \in G.
\]
Let
\[
G_4:=C_4\Gamma_{10}^{-1}G_1=C_4C_2^{-1}F_{18}C_{12}^{-1} \in G,
\]
then also
\[
G_5:=T^{-1}G_4T=B_4B_{2}^{-1}F_{17}B_{12}^{-1} \in G.
\]
We also have 
\[
G_6:=(G_4G_5)G_3(G_4G_5)^{-1}=B_2A_{1}^{-1}F_{14}C_{10}^{-1} \in G.
\]
Hence, $G_3G_{6}^{-1}=C_2B_{2}^{-1} \in G$. By
conjugating with powers of $T$, for $i=1,\ldots,r-1$, we get $C_iB_{i}^{-1} \in G$.
We also have
\begin{eqnarray*}
G_7&:=&T^{-2}G_4T^{2}=C_3C_{1}^{-1}F_{16}C_{11}^{-1} \in G,\\
G_8&:=&T^{4}G_7T^{-4}=C_5C_{3}^{-1}F_{20}C_{13}^{-1} \in G, \textrm{ if } g\geq 29,\\
(G_8&:=&T^{4}G_7T^{-4}=C_5C_{3}^{-1}F_{20}D_{1}^{-1} \in G, \textrm{ if } g=27,)\\
G_9&:=&(G_7G_5)G_8(G_7G_5)^{-1}=C_5B_{4}^{-1}F_{20}C_{13}^{-1} \in G.\textrm{ if } g\geq 29.\\
(G_9&:=&(G_7G_5)G_8(G_7G_5)^{-1}=C_5B_{4}^{-1}F_{20}D_{1}^{-1} \in G, \textrm{ if } g=27.)\\
\end{eqnarray*}
From these, we conclude that $G_8^{-1}G_9=C_3B_{4}^{-1} \in G$. Furhermore, we have 
\[
T^{-4}(B_3C_3^{-1})(C_3B_4^{-1})T^4=B_1B_2^{-1} \in G.
\]
We then have
\begin{eqnarray*}
G_{10}&:=&T^{-4}G_9T^{4}=C_3B_{2}^{-1}F_{16}C_{11}^{-1} \in G,\\
G_{11}&:=&G_{10}(B_2C_2^{-1})=C_3C_{2}^{-1}F_{16}C_{11}^{-1} \in G,\\
G_{12}&:=&T^{-3}G_{11}T^{3}=B_2B_{1}^{-1}F_{13}B_{10}^{-1} \in G.
\end{eqnarray*} 
Thus, we have
\[
T^{5}(B_1B_2^{-1}G_{12})T^{-5}=T^{5}F_{13}B_{10}^{-1}T^{-5}=F_{18}C_{12}^{-1} \in G.
\]
Then, we get the element
\[
G_1F_{18}^{-1}C_{12}=\Gamma_{10}C_2^{-1} \in G,
\]
which implies that $\Gamma_{8}C_1^{-1} \in G$ by conjugating $\Gamma_{10}C_2^{-1}$ with $T^{-2}$. Hence, we have 
\[
C_4C_{2}^{-1}=(C_4\Gamma_{10}^{-1})(\Gamma_{10}C_2^{-1})\in G.
\]
Also, $C_1C_{3}^{-1}$ is in the subgroup $G$ by conjugating $C_2C_{4}^{-1}$ with $T^{-2}$. Then, we have the element
\[
\Gamma_{8}B_4^{-1}=(\Gamma_{8}C_1^{-1})(C_1C_{3}^{-1})(C_3B_{4}^{-1}) \in G.
\]
This implies that $\Gamma_1A_1^{-1}=A_2A_1^{-1}$ is contained in $G$ by conjugating the element $\Gamma_{8}B_4^{-1}$ 
with $T^{-7}$. We conclude that the elements $A_1A_2^{-1}$, $B_1B_2^{-1}$ belong to $G$. By the proof of Theorem~\ref{t1} 
(see \cite{apy1}), the subgroup $G$ contains the elements $A_1$, $A_2$, $B_i$ and $C_j$ for all $i=1,\ldots,r$ and $j=1,\ldots,r-1$. Hence, we get 
$\Gamma_{10}$,$C_{2}$ and $C_{12}$ are in $G$ by conjugating the elements $A_1$ and $A_2$ with powers of $T$. 
We then have
\[
F_1=T^{-17}F_{18}T^{17}=T^{-17}(C_{12}C_2\Gamma_{10}^{-1}G_1)T^{17}  \in G.
\]
It follows from $A_1F_1A_1^{-1}=E$ that the element $E$ is in $G$, which finishes the proof.
\end{proof}
\begin{figure}[h]
\begin{center}
\scalebox{0.22}{\includegraphics{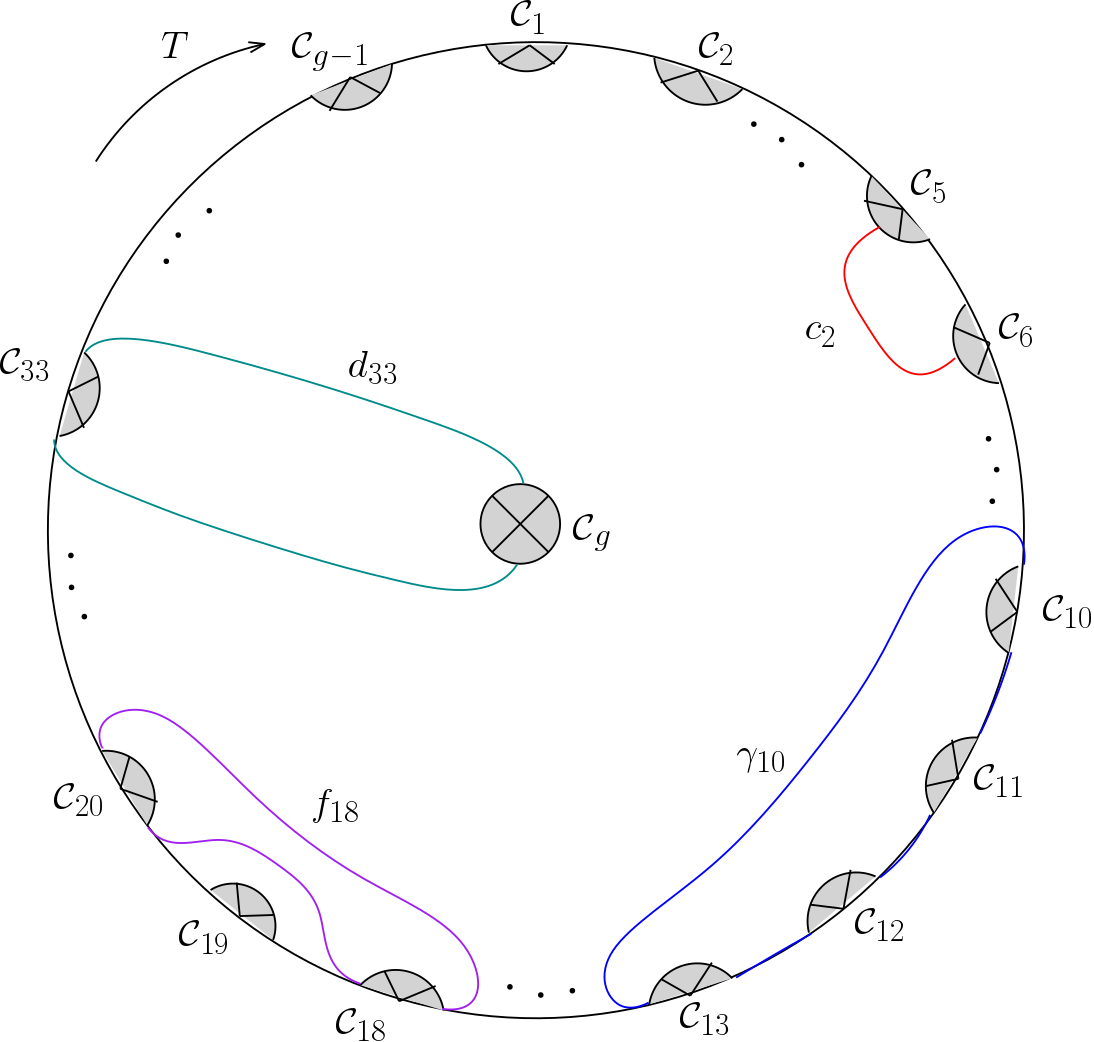}}
\caption{The rotation $T$ of $N_g$ if $g=2r+2$.}
\label{TE}
\end{center}
\end{figure}

\begin{theorem}\label{t42}
For $g=2r+2\geq42$, the twist subgroup $\mathcal{T}_g$ is generated by the elements $T$ and $\Gamma_{10}C_2^{-1}F_{18}D_{33}^{-1}$.
\end{theorem}
\begin{proof}
Consider the surface $N_g$ as in Figure~\ref{TE}.
Let $H_1:=\Gamma_{10}C_2^{-1}F_{18}D_{33}^{-1}$ and let $H$ be the subgroup of $\mathcal{T}_g$ generated by $T$ and $H_1$. It follows from 
Theorem~\ref{t1} that it is enough to show that the elements $A_1A_{2}^{-1}$, $B_1B_{2}^{-1}$, $D_{g-1}$ and $E$ belong to the subgroup $H$. 
By conjugating $H_1$ with $T^{11}$, we get
\[
H_2:=T^{11}H_1T^{-11}=\Gamma_{21}B_8^{-1}F_{29}D_{44}^{-1} \in H.
\]
\begin{eqnarray*}
(H_2:=T^{11}H_1T^{-11}=\Gamma_{21}B_8^{-1}F_{29}D_{3}^{-1} \in H \textit{ if } g=42.)\\
(H_2:=T^{11}H_1T^{-11}=\Gamma_{21}B_8^{-1}F_{29}D_{1}^{-1} \in H \textit{ if } g=44.)
\end{eqnarray*}
Let
\[
H_3:=(H_2H_1)H_2(H_2H_1)^{-1}=\Gamma_{21}B_8^{-1}F_{29}D_{33}^{-1} \in H.
\]
Then, the conjugation of $H_3$ with $T$ gives the element
\[
H_4:=\Gamma_{22}C_8^{-1}F_{30}D_{34}^{-1} \in H.
\]
We then have the following element in $H$ by conjugating $H_4$ with $H_4H_1$:
\[
H_5:=(H_4H_1)H_4(H_4H_1)^{-1}=\Gamma_{22}C_8^{-1}F_{30}D_{33}^{-1}.
\]
Therefore, the element $H_4^{-1}H_5=D_{34}D_{33}^{-1}$ is contained in the subgroup $H$. This implies that the elements $D_{k+1}D_{k}^{-1} \in H$ 
by conjugating with powers of $T$ . It follows that $D_{i}D_{j}^{-1} \in H$ for all $i,j$.

The subgroup $H$ contains the following elements
\begin{eqnarray*}
H_6&:=&(T^{-4}H_1T^{4})D_{29}D_{33}^{-1}=(\Gamma_6 A_{1}^{-1}F_{14}D_{29}^{-1})D_{29}D_{33}^{-1}=\Gamma_6 A_{1}^{-1}F_{14}D_{33}^{-1},\\
H_7&:=&(H_6H_1^{-1})H_6(H_6H_1^{-1})^{-1}=C_2A_{1}^{-1}F_{14}D_{33}^{-1}. 
\end{eqnarray*}
Thus, $H_6H_7^{-1}=\Gamma_6C_2^{-1} \in H$. Moreover, $\Gamma_{10}C_4^{-1}=T^4(\Gamma_6C_2^{-1})T^{-4} \in H$.

We have also the elements
\begin{eqnarray*}
H_8&:=&(C_4\Gamma_{10}^{-1})H_1=(C_4\Gamma_{10}^{-1})(\Gamma_{10}C_2^{-1}F_{18}D_{33}^{-1})=C_4C_2^{-1}F_{18}D_{33}^{-1} \in H,\\
H_9&:=&(T^{-1}H_8T)D_{32}D_{33}^{-1}=(B_4B_2^{-1}F_{17}D_{32}^{-1})(D_{32}D_{33}^{-1})=B_4B_2^{-1}F_{17}D_{33}^{-1} \in H.
\end{eqnarray*}
The conjugation of $H_7$ with $H_8H_9$ gives rise to the following element
\[
H_{10}:=(H_8H_9)H_7(H_8H_9)^{-1}=B_2A_1^{-1}F_{14}D_{33}^{-1}\in H.
\]
Now, using the elements $H_{10}$ and $H_7$, we get $H_{10}H_7^{-1}=B_2C_2^{-1} \in H$ implying that $T(B_2C_2^{-1})T^{-1}=C_2B_3^{-1} \in H$. We then have the element
\[
B_1B_2^{-1}=T^{-2}\big((B_2C_2^{-1})(C_2B_3^{-1})\big)T^2=T^{-2}(B_2B_3^{-1})T^2 \in H.
\]

We easily see that the elements 
\begin{eqnarray*}
H_{11}&:=&(T^{-2}H_{8}T^2)(D_{31}D_{33}^{-1})=(C_3C_1^{-1}F_{16}D_{31}^{-1})(D_{31}D_{33}^{-1})=C_3C_1^{-1}F_{16}D_{33}^{-1},\\
H_{12}&:=&(T^{4}H_{11}T^{-4})(D_{37}D_{33}^{-1})=(C_5C_3^{-1}F_{20}D_{37}^{-1})(D_{37}D_{33}^{-1})=C_5C_3^{-1}F_{20}D_{33}^{-1}
\end{eqnarray*}
are in the subgroup $H$. The conjugation of $H_{12}$ with the element $H_{11}H_9$ yield the following element 
\begin{eqnarray*}
H_{13}:=(H_{11}H_{9})H_{12}(H_{11}H_{9})^{-1}=C_5B_4^{-1}F_{20}D_{33}^{-1} \in H.
\end{eqnarray*}
Moreover, we have the elements
\begin{eqnarray*}
H_{14}&:=&(T^{-4}H_{13}T^4)D_{29}D_{33}^{-1}=(C_3B_2^{-1}F_{16}D_{29}^{-1})D_{29}D_{33}^{-1}=C_3B_2^{-1}F_{16}D_{33}^{-1},\\
H_{15}&:=&H_{14}(B_2C_2^{-1})=C_3B_2^{-1}F_{16}D_{33}^{-1}(B_2C_2^{-1})=C_2^{-1}C_3F_{16}D_{33}^{-1} \textrm{ and } \\
H_{16}&:=&(T^{-3}H_{15}T^3)D_{30}D_{33}^{-1}=(B_1^{-1}B_2F_{13}D_{30}^{-1})D_{30}D_{33}^{-1}=B_1^{-1}B_2F_{13}D_{33}^{-1}
\end{eqnarray*}
are contained in the subgroup $H$. Hence, the subgroup $H$ contains the element 
\[
F_{13}D_{33}^{-1}=(B_1B_2^{-1})H_{16}=(B_1B_2^{-1})B_1^{-1}B_2F_{13}D_{33}^{-1}
\]
which implies that 
\[
F_{18}D_{33}^{-1}= \big(T^{5}(F_{13}D_{33}^{-1})T^{-5}\big)D_{38}D_{33}^{-1}\in H.
\]
From this, we have
\[
H_1(F_{18}^{-1}D_{33})=\Gamma_{10}C_2^{-1} \in H
\]
which implies that the element
\[
C_4C_2^{-1}=(C_4\Gamma_{10}^{-1})(\Gamma_{10}C_2^{-1}) \in H.
\]
The conjugation of $C_4C_2^{-1}$ with $T^{2}$ gives the element $C_5C_3^{-1} \in H$. Thus implies that 
\[
B_4C_5^{-1}=(B_4C_3^{-1})(C_3C_5^{-1}) \in H.
\]
Then, we have the element
\[
\Gamma_8B_4^{-1}=(\Gamma_8C_1^{-1})(C_1C_3^{-1})(C_3C_5^{-1})(C_5B_4^{-1}) \in H.
\]
Hence the conjugation of $\Gamma_8B_4^{-1}$ with $T^{-7}$ is the element $\Gamma_1A_1^{-1}=A_2A_1^{-1}$, which belongs to $H$. By the proof of Theorem~\ref{t1}, 
we conclude that the elements $A_1$, $A_2$, $B_i$ and $C_j$ are all in the subgroup $H$ for all $i=1,\ldots,r$ and $j=1,\ldots,r-1$. It follows from
\[
F_{18}^{-1}D_{33}B_{16}(f_{18},d_{33})=(f_{18},b_{16})
\]
that $F_{18}B_{16}^{-1} \in H$. This implies that $F_{18}$ is contained in $H$ since $B_{16}\in H$. Then, we get the elements $F_1=T^{-17}F_{18}T^{17} \in H$ and 
$A_1F_1A_1^{-1}=E\in H$. Also, $D_{33}=F_{18}(F_{18}^{-1}D_{33}) \in H$. Finally, the element $D_{g-1}$ belongs to $H$ by conjugating 
$D_{33}$ with $T^{-33}$, which finishes the proof.
\end{proof}
If one wants to decrease $g$, then the number of generators increases. In the remaining parts of this section, we try to find minimal number of generators for smaller $g$.

\begin{theorem}\label{t9odd}
For $g=2r+1\geq 9$, the twist subgroup $\mathcal{T}_g$ is generated by the elements $T$, $A_1A_{2}^{-1}$ and $F_1B_{2}^{-1}$.
\end{theorem}
\begin{proof}
Consider the surface $N_g$ as shown in Figure~\ref{TO}. Let $G$ denote the subgroup of $\mathcal{T}_g$ generated by the elements $T$, $A_1A_{2}^{-1}$ and $F_1B_{2}^{-1}$. By Theorem~\ref{t1}, it suffices to prove that the elements $B_1B_2^{-1}$ and $E$ are contained in the subgroup $G$. It follows from the element $T^{-3}$ maps the curves $(f_1,b_2)$ to the curves $(f_{g-2},a_1)$ that we obtain
\[
T^{-3}(F_1B_2^{-1})T^{3}=F_{g-2}A_1^{-1}.
\]
Since each factor on the left-hand side is contained in $G$, the element $F_{g-2}A_1^{-1}$ is also contained in $G$. Also, since 
\[
(A_1F_{g-2}^{-1})(F_1B_2^{-1})(f_{g-2},a_1)=(f_{g-2},f_1),
\]
the subgroup $G$ contains the element $F_{g-2}F_1^{-1}$. Hence, we get the element 
\[
F_{g-2}B_2^{-1}=(F_{g-2}F_1^{-1})(F_1B_2^{-1}) \in G.
\]
From these, one obtains the following elements
\begin{eqnarray*}
B_2A_1^{-1}&=&(B_2F_{g-2}^{-1})(F_{g-2}A_1^{-1}) \in G,\\
B_2A_2^{-1}&=&(B_2A_1^{-1})(A_{1}A_2^{-1}) \in G.
\end{eqnarray*}
Then, the element
\[
B_1\Gamma_{g-1}^{-1}=T^{-2}(B_2A_2^{-1})T^2\in G,
\]
since $T^{-2}$ maps $(b_2,a_2)$ to $(b_1,\gamma_{g-1})$. Then, it follows from
\[
(A_1B_2^{-1})(B_1\Gamma_{g-1}^{-1})(a_1,b_2)=(b_1,b_2)
\]
 and $A_1B_2^{-1} \in G$ that the element $B_1B_2^{-1} \in G$.
Since the subgroup $G$ contains the elements $T$, $A_1A_2^{-1}$ and $B_1B_2^{-1}$, the proof of Theorem~\ref{t1} (see \cite{apy1}) implies that the elements $A_1$, $A_2$, $B_i$ and $C_j$ belong to the subgroup $G$ for all $i=1,\ldots,r$ and $j=1,\ldots,r-1$. In particular $B_2 \in G$, which implies that $F_1=(F_1B_2^{-1})B_2 \in G$. This completes the proof since the element $E=A_1F_1A_1^{-1} \in G$.
\end{proof}

\begin{theorem}\label{t8even}
For $g=2r+2\geq 8$, the twist subgroup $\mathcal{T}_g$ is generated by the elements $T$, $D_{g-1}A_{2}^{-1}$ and $F_1B_{2}^{-1}$.
\end{theorem}
\begin{proof}
Let us consider the surface $N_g$ as in Figure~\ref{TE} and 
let $H$ be the subgroup of $\mathcal{T}_g$ generated by $T$, $D_{g-1}A_{2}^{-1}$ and $F_1B_{2}^{-1}$. As in the proof of Theorem~\ref{t42}, it is enough to show that the generators  $A_1A_{2}^{-1}$, $B_1B_{2}^{-1}$, $D_{g-1}$ and $E$ are contained in the subgroup $H$. It follows from 
\[
T^{-3}(f_1,b_2)=(f_{g-3},a_1)
\]
and
\[
(F_{g-3}A_1^{-1})(D_{g-1}A_2^{-1})(f_{g-3},a_1)=(d_{g-1},a_1)
\]
that the element $F_{g-3}A_1^{-1} \in H$ and so
$D_{g-1}A_1^{-1}\in H$
Hence, the subgroup $H$ contains the element
\[
A_1A_2^{-1}=(A_1D_{g-1}^{-1})(D_{g-1}A_2^{-1}).
\]
Then, since the subgroup $H$ contains the element $(A_1D_{g-1}^{-1})(F_1B_2^{-1})$ which takes the curves $(a_1,d_{g-1})$ to $(f_1,d_{g-1})$, we have $F_1D_{g-1}^{-1}\in H$.
This implies that 
\[
F_1A_1^{-1}=(F_1D_{g-1}^{-1})(D_{g-1}A_1^{-1}) \in H
\]
and so
\[
A_1B_2^{-1}=(A_1F_1^{-1})(F_1B_2^{-1})\in H.
\]
By conjugating the element $A_1B_2^{-1}$ with $T^2$, we see that $C_1B_3^{-1}$ belongs to $H$. Then, it follows from 
\[
(C_1B_3^{-1})(B_2A_1^{-1})(b_3,c_1)=(b_3,b_2)
\]
that $H$ contains the element $B_3B_2^{-1}$. From this,
\[
T^{-2}(B_3B_2^{-1})T^2=B_2B_1^{-1} \in H.
\]
Thus, $H$ contains the elements $T$, $A_1A_2^{-1}$ and $B_1B_2^{-1}$, which implies that it also contains the Dehn twists $A_1$, $A_2$, $B_1,\ldots,B_r$ and $C_1,\ldots,C_{r-1}$ by the proof of Theorem~\ref{t1}. Hence, we have
\[
D_{g-1}=(D_{g-1}A_2^{-1})A_2\in H
\]
and 
\[
E=A_1F_1A_1^{-1}\in H.
\]
Therefore, we conclude that $\mathcal{T}_g=H$.
\end{proof}

\section{Commutator generators for $\mathcal{T}_g$}\label{S4}
In this section, for $g=4k$ or $g=4k+1$ we use the generators obtained in the previous section since each generator can be expressed as a single commutator. For these cases, number of generators and commutator generators coincide. On the other hand, for $g=4k+2$ or $g=4k+3$ using our methods, we cannot express the generators of the previous section as a single commutator. For these cases, we come up with generating sets where each element can be expressed as a single commutator.

For a nonorientable genus $g$ surface $N_g$, let us introduce the reflections $\rho_1$ and $\rho_2$ which are contained in the twist subgroup $\mathcal{T}_g$.
\begin{figure}[h]
\begin{center}
\scalebox{0.2}{\includegraphics{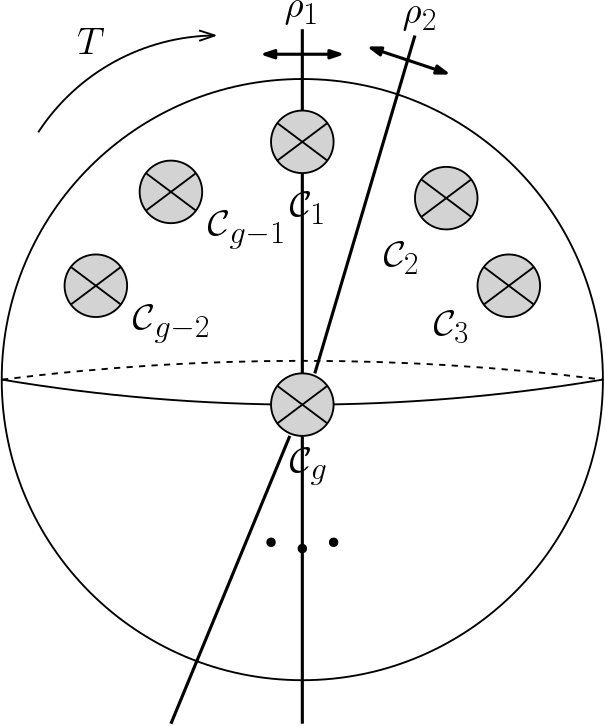}}
\caption{The reflections $\rho_1$ and $\rho_2$ on the surface $N_g$ for $g=4k$.}
\label{R4K}
\end{center}
\end{figure}

\underline{If $g=4k$}, consider the reflections $\rho_1$ and $\rho_2$ in the indicated planes as shown in Figure~\ref{R4K} such that the rotation $T$ is given by $T=\rho_2\rho_1$. Thus
\begin{enumerate}
\item[$\bullet$]$T(x_i)=x_{i+1}$ for $i=1,\ldots,g-2$,
\item[$\bullet$] $T(x_{g-1})=x_{1}$ and $T(x_g)=x_{g}$.
\end{enumerate}
(Recall that the curves $x_i$`s are the generators of $H_1(N_g;\mathbb{R})$ as shown in Figure~\ref{H}.) It can be shown that the reflections $\rho_1$ and $\rho_2$ satisfy $D(\rho_1)=D(\rho_2)=1$, which implies that $\rho_1$ and $\rho_2$ are contained in the twist subgroup $\mathcal{T}_g$ if $g=4k$.

\begin{figure}[h]
\begin{center}
\scalebox{0.2}{\includegraphics{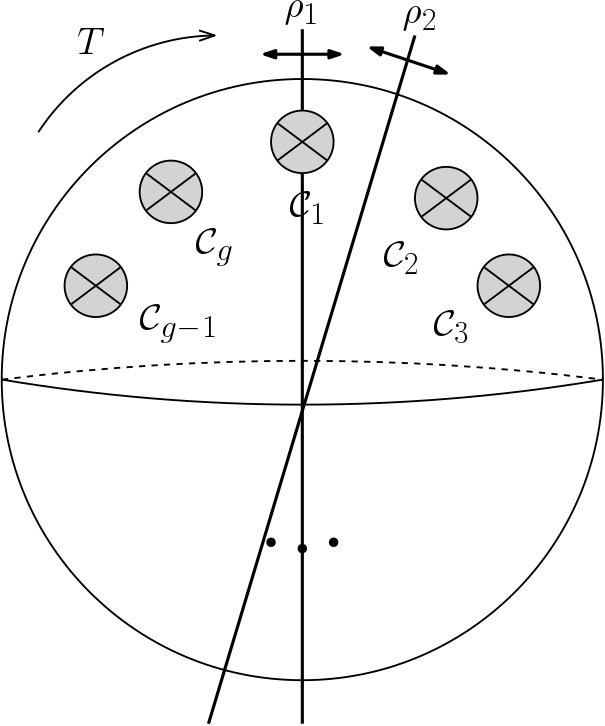}}
\caption{The reflections $\rho_1$ and $\rho_2$ on the surface $N_g$ for $g=4k+1$.}
\label{R4K+1}
\end{center}
\end{figure}

\underline{If $g=4k+1$}, then consider the reflections $\rho_1$ and $\rho_2$ in the planes depicted in Figure~\ref{R4K+1} so that the rotation $T$ can be expressed as $T=\rho_2\rho_1$. Thus
\begin{enumerate}
\item[$\bullet$]$T(x_i)=x_{i+1}$ for $i=1,\ldots,g-1$ and $T(x_{g})=x_{1}$.
\end{enumerate}
 It can be verified that the reflections $\rho_1$ and $\rho_2$ are contained in the twist subgroup $\mathcal{T}_g$ since they satisfy $D(\rho_1)=D(\rho_2)=1$ if $g=4k+1$.

\begin{figure}[h]
\begin{center}
\scalebox{0.27}{\includegraphics{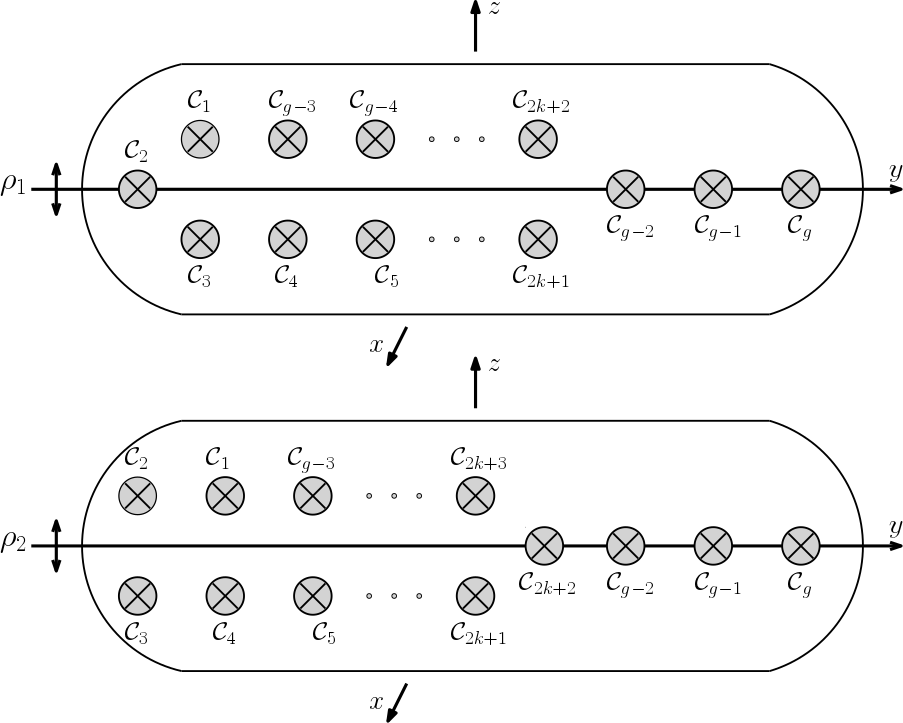}}
\caption{The reflections $\rho_1$ and $\rho_2$ on the surface $N_g$ for $g=4k+2$.}
\label{R4K+2}
\end{center}
\end{figure}

\underline{If $g=4k+2$}, consider the reflections $\rho_1$ and $\rho_2$ in the planes as shown in Figure~\ref{R4K+2} so that $T=\rho_2\rho_1$. Hence
\begin{enumerate}
\item[$\bullet$]$T(x_i)=x_{i+1}$ for $i=1,\ldots,g-4$, $T(x_{g-3})=x_{1}$,
\item[$\bullet$] $T(x_j)=x_j$ for $j\in \lbrace g-2,g-1,g \rbrace$.
\end{enumerate}
 It is easy to see that the reflections $\rho_1$ and $\rho_2$ satisfy $D(\rho_1)=D(\rho_2)=1$ if $g=4k+2$. Hence they are contained in the twist subgroup $\mathcal{T}_g$ if $g=4k+2$.

\begin{figure}[h]
\begin{center}
\scalebox{0.2}{\includegraphics{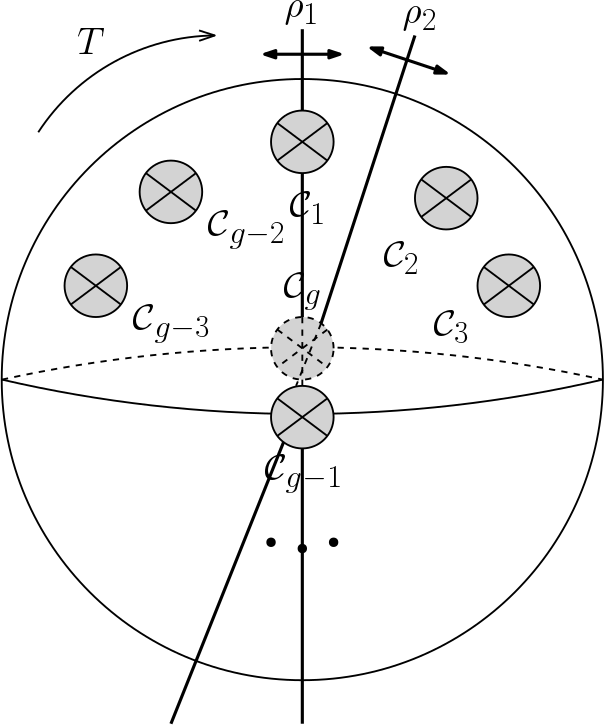}}
\caption{The reflections $\rho_1$ and $\rho_2$ on the surface $N_g$ for $g=4k+3$.}
\label{R4K+3}
\end{center}
\end{figure}

\underline{If $g=4k+3$}, then consider the reflections $\rho_1$ and $\rho_2$ in the planes as depicted in Figure~\ref{R4K+3} so that $T=\rho_2\rho_1$. Hence
\begin{enumerate}
\item[$\bullet$]$T(x_i)=x_{i+1}$ for $i=1,\ldots,g-3$, $T(x_{g-2})=x_{1}$,
\item[$\bullet$] $T(x_j)=x_j$ for $j=g-1,g$.
\end{enumerate}
 It can be seen that  $\rho_1$ and $\rho_2$ satisfy $D(\rho_1)=D(\rho_2)=1$ if $g=4k+3$. Hence, $\mathcal{T}_g$ contains the elements  $\rho_1$ and $\rho_2$ if $g=4k+3$.

\begin{proposition}\label{prop}
The mapping class $T$ in the twist subgroup $\mathcal{T}_g$ can be expressed as
\begin{enumerate}
\item[$(1)$] $T=[T^{2k},\rho_1]$, if $g=4k$ or $g=4k+2$ and
\item[$(2)$]  $T=[T^{2k+1},\rho_1]$, if $g=4k+1$ or $g=4k+3$ ,
\end{enumerate}
where $k\geq1$.
\end{proposition}
\begin{proof}
Let $g=4k$ or $g=4k+2$, then consider the models for $N_g$ in Figure~\ref{R4K} or Figure~\ref{R4K+2}, respectively so that $T=\rho_2\rho_1$. It can be verified that 
\[
\rho_2=T^{2k}\rho_1T^{-2k}.
\]
Therefore, $T=\rho_2\rho_1=T^{2k}\rho_1T^{-2k}\rho_1=[T^{2k},\rho_1]$, which proves $(1)$.

For $(2)$, Let $g=4k+1$ or $g=4k+3$, let us consider the models for $N_g$ in Figure~\ref{R4K+1} or Figure~\ref{R4K+3}, respectively in such a way that $T=\rho_2\rho_1$. Hence, we have
\[
\rho_2=T^{2k+1}\rho_1T^{-(2k+1)}.
\]
Therefore, $T=\rho_2\rho_1=T^{2k+1}\rho_1T^{-(2k+1)}\rho_1=[T^{2k+1},\rho_1]$.
\end{proof}
Now, we give new generating sets for $\mathcal{T}_g$ when $g=4k+2$ and $g=4k+3$.
\begin{figure}[h]
\begin{center}
\scalebox{0.25}{\includegraphics{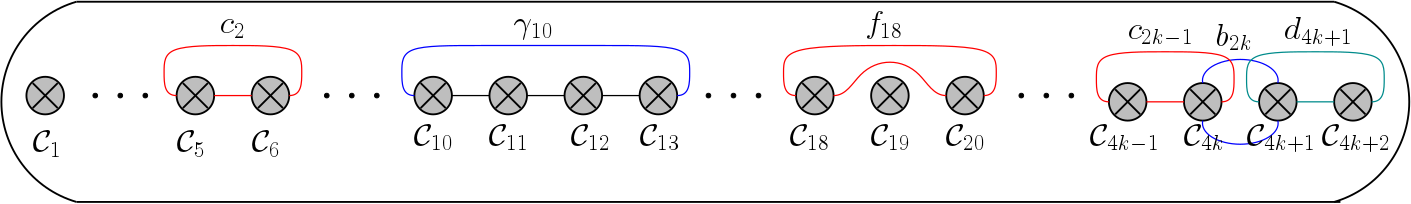}}
\caption{The curves $c_2,\gamma_{10},f_{18},c_{2k-1},b_{2k}$ and $d_{4k+1}$ on $N_g$ for $g=4k+2$.}
\label{Ck+2}
\end{center}
\end{figure}

In the following theorem, we give a generating set for $\mathcal{T}_g$ consisting of three elements when $g=4k+2$, where the first generator $T=\rho_2\rho_1$ is shown in Figure~\ref{R4K+2}. The Dehn twist curves that are contained in the generating set are shown in Figure~\ref{Ck+2}. Note that $T$ satisfies 
\begin{enumerate}
\item[$\bullet$] $T(b_i)=c_i$, $T(c_j)=b_{j+1}$ for $i,j=1,\ldots,2k-2$,
\item[$\bullet$] $T(a_1)=(b_1)$, $T(b_{2k})=b_{2k}$ and $T(d_{4k+1})=d_{4k+1}$.
\end{enumerate}

\begin{theorem}\label{t4k+2}
For $g=4k+2\geq30$, the twist subgroup $\mathcal{T}_g$ is generated by the three elements $T$, $\Gamma_{10}C_{2}^{-1}F_{18}B_{2k}^{-1}$ and $C_{2k-1}D_{4k+1}^{-1}$.
\end{theorem}
\begin{proof}
Let $K_1:=\Gamma_{10}C_{2}^{-1}F_{18}B_{2k}^{-1}$ and let us denote by $K$ the subgroup of $\mathcal{T}_g$ generated by the elements $T$, $K_1$ and $C_{2k-1}D_{4k+1}^{-1}$. First, we need to show that $A_1A_2^{-1}$ and $B_1B_2^{-1}$ are contained in the subgroup $K$. This implies that $A_1,A_2$, $B_i,C_j \in K$ for $i=1,\ldots,2k-1$ and $j=1,\ldots,2k-2$ by the proof of Theorem~\ref{t1} and the action of $T$ on $N_g$. To finish the proof, we also need to prove that the subgroup $K$ contains the Dehn twists $C_{2k-1}$, $B_{2k}$, $D_{4k+1}$ and $E$.

Let
\begin{eqnarray*}
K_2&:=&T^{-4}K_1T^{4}=\Gamma_{6}A_1^{-1}F_{14}B_{2k}^{-1} \in K \textrm{ and }\\
K_3&:=&(K_2K_1^{-1})K_2(K_2K_1^{-1})^{-1}=C_2A_1^{-1}F_{14}B_{2k}^{-1} \in K.
\end{eqnarray*}
Then, $K_2K_3^{-1}=\Gamma_6C_2^{-1}$ and also $T^{4}(\Gamma_6C_2^{-1})T^{-4}=\Gamma_{10}C_4^{-1}\in K$.

We also have the element
\begin{eqnarray*}
K_4&:=&(C_4\Gamma_{10}^{-1})K_1=C_4C_{2}^{-1}F_{18}B_{2k}^{-1}\in K,\\
K_5&:=&T^{-1}K_4T=B_4B_2^{-1}F_{17}B_{2k}^{-1}\in K \textrm{ and }\\
K_6&:=&(K_4K_5)K_3(K_4K_5)^{-1}=B_2A_1^{-1}F_{14}B_{2k}^{-1} \in K.
\end{eqnarray*}
Thus, $K_3K_6^{-1}=C_2B_2^{-1}\in K$ implying that $T(B_2C_2^{-1})T^{-1}=C_2B_3^{-1} \in K$. We then get 
\[
B_1B_2^{-1}=T^{-2}\big((B_2C_2^{-1})(C_2B_3^{-1})\big)T^2=T^{-2}(B_2B_3^{-1})T^2 \in K.
\]

We also have the following elements:
\begin{eqnarray*}
K_7&:=&T^{-2}K_4T^2=C_3C_1^{-1}F_{16}B_{2k}^{-1},\\
K_8&:=&T^4K_7T^{-4}=C_5C_3^{-1}F_{20}B_{2k}^{-1} \textrm{ and }\\
K_9&:=&(K_7K_5)K_8(K_7K_5)^{-1}=C_5B_4^{-1}F_{20}B_{2k}^{-1}
\end{eqnarray*}
are all in $K$. 
Similarly, the subgroup $K$ contains the following elements:
\begin{eqnarray*}
K_{10}&:=&T^{-4}K_9T^4=C_3B_2^{-1}F_{16}B_{2k}^{-1},\\
K_{11}&:=&K_{10}(B_2C_2^{-1})=C_2^{-1}C_3F_{16}B_{2k}^{-1} \textrm{ and }\\
K_{12}&:=&T^{-3}K_{10}T^3=B_1^{-1}B_2F_{13}B_{2k}^{-1}.
\end{eqnarray*}
 Hence, $K_{12}(B_1B_2^{-1})=F_{13}B_{2k}^{-1}\in K$. This implies that $T^{5}(F_{13}B_{2k}^{-1})T^{-5}=F_{18}B_{2k}^{-1}$ is contained in $K$. Using this, we conclude that $K_1(B_{2k}F_{18}^{-1})=\Gamma_{10}C_2^{-1} \in K$. From this, the subgroup $K$ also contains the element $T^{-2}(\Gamma_{10}C_2^{-1} )T^{2}=\Gamma_8C_1^{-1}$.

It follows from $(C_4 \Gamma_{10}^{-1})(\Gamma_{10}C_2^{-1})=C_4C_2^{-1}\in K$ that the elements $C_3C_1^{-1}$ and $C_5C_3^{-1}$ belong to $K$ by conjugating with $T^{-2}$ and $T^2$, respectively. 

Moreover, we have $(B_4C_3^{-1})(C_3C_5^{-1})=B_4C_5^{-1}\in K$. Hence,
\[
(\Gamma_8C_1^{-1})(C_1C_3^{-1})(C_3C_5^{-1})(C_5B_4^{-1})=\Gamma_8B_4^{-1}\in K,
\]
which implies that $T^{-7}(\Gamma_8B_4^{-1})T^7=A_2A_1^{-1}\in K$. We conclude that $A_1,A_2$, $B_i,C_j \in K$ for $i=1,\ldots,2k-1$ and $j=1,\ldots,2k-2$ by the proof of Theorem~\ref{t1}.

The subgroup $K$ contains the element $(C_{2k-1}D_{4k+1}^{-1})(B_{2k-1})$ which maps the curves $(c_{2k-1},d_{4k+1})$ to $(b_{2k-1},d_{4k+1})$. Hence, the subgroup $K$ contains $B_{2k-1}D_{4k+1}^{-1}$ implying that $D_{4k+1}=D_{g-1}\in K$ by the fact that $B_{2k-1}\in K$. Also, we obtain the element $C_{2k-1} \in K$.

Moreover, since the subgroup $K$ has the elements $F_{18}B_{2k}^{-1}$ and $D_{4k+1}$ and the element $(F_{18}B_{2k}^{-1})(D_{4k+1}^{-1})$ sends the curves $(f_{18},b_{2k})$ to $(f_{18},d_{4k+1})$, we get the element $F_{18}D_{4k+1}^{-1}\in K$. We then conclude that $F_{18}$ belongs to $K$, which also implies that $B_{2k}$ is contained in $K$. Finally, since $T^{-17}F_{18}T^{17}=F_1\in K$, the element $A_1F_1A_1^{-1}=E$ is in $K$. This completes the proof.
\end{proof}
\begin{figure}[h]
\begin{center}
\scalebox{0.22}{\includegraphics{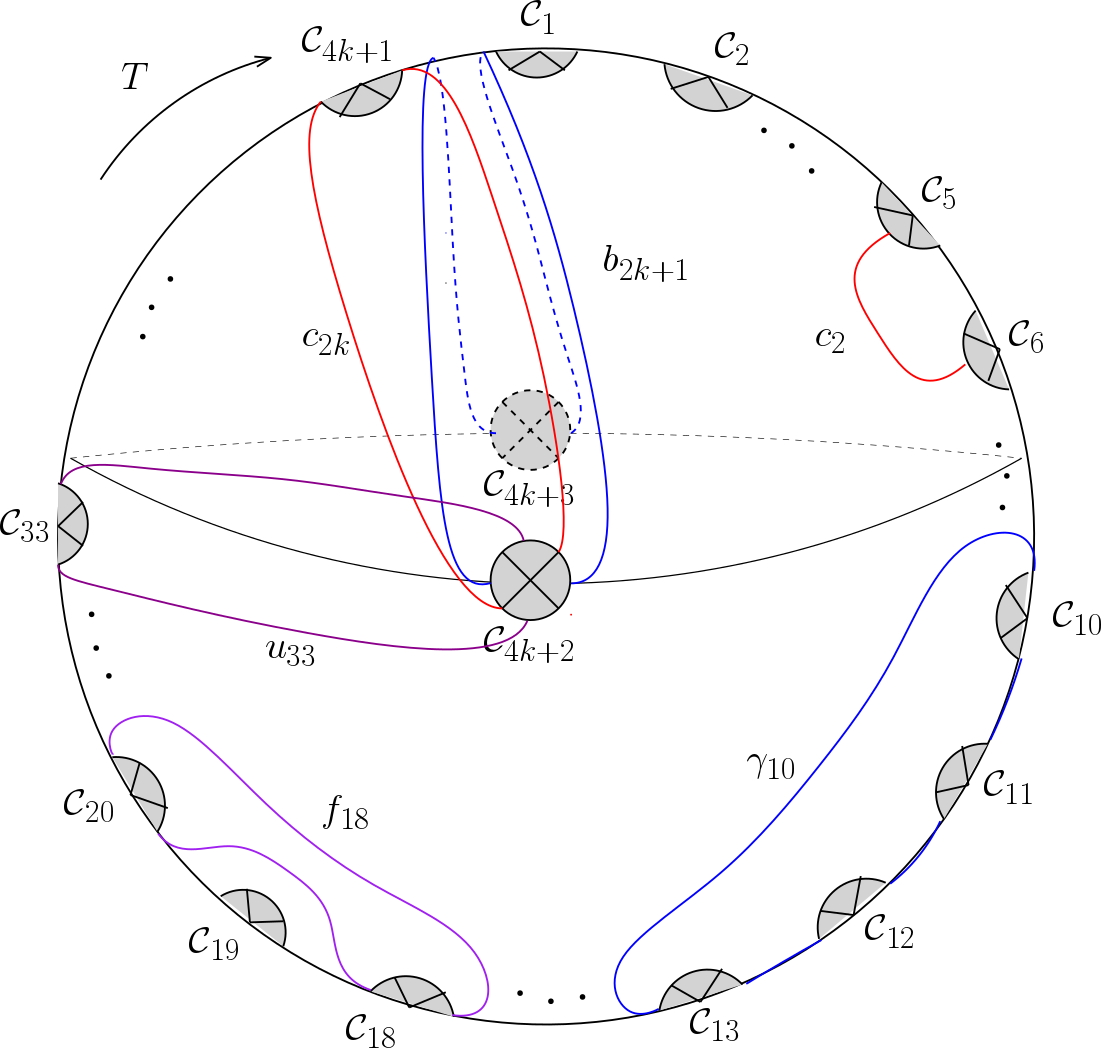}}
\caption{The rotation $T$ and the curves $c_2,\gamma_{10},f_{18},u_{33},c_{2k}$ and $b_{2k+1}$ on $N_g$ for $g=4k+3$.}
\label{Ck+3}
\end{center}
\end{figure}

The following theorem provides a generating set for $\mathcal{T}_g$ consisting of three mapping classes when $g=4k+3$, 
where the first generator $T=\rho_2\rho_1$ is shown in Figure~\ref{R4K+3}. The other generators contain the Dehn twist about 
the curves shown in Figure~\ref{Ck+3} (note that $c_{2k}=u_{4k+1}$). In this case, the rotation $T$ satisfies 
\begin{enumerate}
\item[$\bullet$]  $T(a_1)=(b_1)$, $T(b_i)=c_i$, $T(c_i)=b_{i+1}$ for $i=1,\ldots,2k-1$ and 
\item[$\bullet$]$T(u_j)=u_{j+1}$ for $j=1,\ldots,4k$ and $T(u_{4k+1})=u_{1}$.
\end{enumerate}
\begin{theorem}\label{t4k+3}
For $g=4k+3\geq43$, the twist subgroup $\mathcal{T}_g$ is generated by the three elements $T$, $\Gamma_{10}C_{2}^{-1}F_{18}U_{33}^{-1}$ and $B_{2k+1}A_{1}^{-1}$.
\end{theorem}
\begin{proof}
Let  $L_1$ denote the mapping class $\Gamma_{10}C_{2}^{-1}F_{18}U_{33}^{-1}$ and $L$ be the subgroup of $\mathcal{T}_{g}$ 
generated by the elements $T$, $L_1$ and $B_{2k+1}A_{1}^{-1}$. It follows from Theorem~\ref{t1} that we need to show that the subgroup 
$L$ contains the elements $A_1A_{2}^{-1}$, $B_1B_{2}^{-1}$, $C_{2k}=U_{4k+1}$, $B_{2k+1}$ and $E$.

We will now mainly follow the steps of the proof of Theorem~\ref{t42}.  We have the elements
\[
L_2:=T^{11}L_1T^{-11}=\Gamma_{21} B_8^{-1}F_{29}U_{44}^{-1}  \textrm{ if } g\geq47,
\]
\begin{eqnarray*}
(L_2:=T^{11}L_1T^{-11}=\Gamma_{21}B_8^{-1}F_{29}U_{3}^{-1} \in H \textrm{ if } g=43),
\end{eqnarray*}
\begin{eqnarray*}
L_3&:=&(L_2L_1)L_2(L_2L_1)^{-1}=\Gamma_{21}B_8^{-1}F_{29}U_{33}^{-1} \\
L_4&:=&TL_4T^{-1}=\Gamma_{22}C_8^{-1}F_{30}U_{34}^{-1} \textrm{ and }\\
L_5&:=&(L_4L_1)L_4(L_4L_1)^{-1}=\Gamma_{22}C_8^{-1}F_{30}U_{33}^{-1},
\end{eqnarray*}
which are all contained in the subgroup $L$. Thus, we get $L_4^{-1}L_5=U_{34} U_{33}^{-1}\in L$ which implies that, by conjugating with powers of $T$,  
the elements $U_{i+1}U_{i}^{-1} \in L$.  It follows  that $U_{i}U_{j}^{-1} \in L$ for all $i, j$.
By applying the same steps in the proof of Theorem~\ref{t42} by taking $U_i$ instead of $D_i$, we get the elements 
$A_2A_1^{-1}$, $B_1B_2^{-1}$ and $F_{18}U_{33}^{-1}$ in the subgroup $L$. This implies that the generators $A_1,A_2, B_i, C_i$ for $i=1,\ldots,2k-1$. Then, it follows from
\[
F_{18}^{-1}U_{33}B_{16}(f_{18},u_{33})=(f_{18},b_{16})
\]
that $F_{18}B_{16}^{-1} \in L$. Then we have $F_{18}\in L$ since $B_{16} \in L$. Thus, we get the elements $F_1=T^{-17}F_{18}T^{17} \in L$ and $A_1F_1A_1^{-1}=E\in L$. Moreover, $U_{33}=F_{18}(F_{18}^{-1}U_{33}) \in L$, which implies that the element $C_{2k}\in L$ by conjugating $U_{33}$ with $T^{4k-32}$. Finally, since $A_1$ and $B_{2k+1}A_1^{-1}$ belong to $L$, $B_{2k+1}\in L$, which finishes the proof.
\end{proof}
If one wants to decrease $g$, then the number of generators increases. We try to find minimal number of generators for $g=4k+2\geq 10$ or $g=4k+3\geq 7$.
\begin{theorem}\label{t4k+2.10}
For $g=4k+2\geq10$, the twist subgroup $\mathcal{T}_g$ is generated by the four elements $T$, $A_1A_2^{-1}$, $F_1B_{2k}^{-1}$ and $C_{2k-1}D_{4k+1}^{-1}$.
\end{theorem}
\begin{proof}
Let $K$ be the subgroup of $\mathcal{T}_g$ generated by the elements $T$, $A_1A_2^{-1}$, $F_1B_{2k}^{-1}$ and $C_{2k-1}D_{4k+1}^{-1}$. As in the proof of Theorem~\ref{t4k+2}, we need to prove that the subgroup $K$ contains the elements $B_1B_2^{-1}$, $C_{2k-1}$, $B_{2k}$, $D_{4k+1}$ and $E$. Recall that the element $T=\rho_2\rho_1$ satisfies $T(b_{2k})=b_{2k}$, where $\rho_1$ and $\rho_2$ are as shown in Figure~\ref{R4K+2}. 

It is easy to check that $T^3(f_1,b_{2k})=(f_4,b_{2k})$. Since $F_1B_{2k}^{-1} \in K$, we get the element
\[
F_4B_{2k}^{-1}=T^3(F_1B_{2k}^{-1})T^{-3}
\]
is contained in $K$. Also, it can be verified that the diffeomorphism $(A_1A_2^{-1})(B_{2k}F_4^{-1})$ maps the curves $(a_1,a_2)$ to $(a_1,f_4)$ so that $K$ contains the element $A_1F_4^{-1}$. From this, we have
\[
A_1B_{2k}^{-1}=(A_1F_4^{-1})(F_4B_{2k}^{-1})\in K.
\]
Now, it follows from $T(a_1,b_{2k})=(b_1,b_{2k})$ and $T^2(b_1,b_{2k})=(b_2,b_{2k})$ that the elements $B_1B_{2k}^{-1}$ and $B_2B_{2k}^{-1}$ are in $K$.
 Hence, we have 
\[
B_1B_2^{-1}=(B_1B_{2k}^{-1})(B_{2k}B_2^{-1}) \in K.
\]
Now , since $T$, $A_1A_{2}^{-1}$, $B_1B_{2}^{-1}$ are in $K$, the proof of Theorem~\ref{t1} implies that the generators $A_1$, $A_2$, $B_1,\ldots,B_{2k-1}$ and $C_1,\ldots,C_{2k-2}$ are all in $K$. Thus, we have
\begin{eqnarray*}
B_{2k}&=&(B_{2k}B_1^{-1})B_1\in K \textrm{ and}\\
E&=&A_1F_1A_1^{-1}=A_1(F_1B_{2k}^{-1})(B_{2k})A_1^{-1} \in K.
\end{eqnarray*}
On the other hand, the subgroup $K$ contains the element $(C_{2k-1}D_{4k+1}^{-1})(B_{2k-1})$ that maps the curves $(c_{2k-1},d_{4k+1})$ to $(b_{2k-1},d_{4k+1})$. This implies that the element $B_{2k-1}D_{4k+1}^{-1} \in K$. Therefore, the elements
\begin{eqnarray*}
D_{4k+1}&=&(D_{4k+1}B_{2k-1}^{-1})(B_{2k-1}) \textrm{ and}\\
C_{2k-1}&=&(C_{2k-1}D_{4k+1}^{-1})(D_{4k+1}),
\end{eqnarray*}
are contained in $K$, which completes the proof.
\end{proof}

\begin{theorem}\label{t4k+3.7}
For $g=4k+3\geq7$, the twist subgroup is generated by the four elements $T$, $A_1A_2^{-1}$, $F_{g-2}U_3^{-1}$ and $B_{2k}B_{2k+1}^{-1}$.
\end{theorem}
\begin{proof}
Let $L$ be the subgroup of $\mathcal{T}_g$ generated by the elements $T$, $A_1A_2^{-1}$, $F_{g-2}U_3^{-1}$ and $B_{2k}B_{2k+1}^{-1}$. It can be verified that the diffeomorphism $(F_{g-2}U_3^{-1})(A_1A_2^{-1})$ takes the curves $(f_{g-2},u_3)$ to $(f_{g-2},a_2)$. This implies that 
$F_{g-2}A_2^{-1} \in L$. Hence we get the elements 
\begin{eqnarray*}
A_1F_{g-2}^{-1}&=&(A_1A_2^{-1})(A_2 F_{g-2}^{-1})\\
U_3A_1^{-1}&=&(U_3F_{g-2}^{-1})(F_{g-2}A_1^{-1})
\end{eqnarray*}
are contained in  $L$. Also, the subgroup $L$ contains the element $(B_{2k}B_{2k+1}^{-1})(A_1U_3^{-1})$ sending the curves $(b_{2k},b_{2k+1})$ to $(b_{2k},u_3)$. Hence we get 
$B_{2k}U_3^{-1} \in L$.
From this, we also have the elements
\begin{eqnarray*}
B_{2k}A_1^{-1}&=&(B_{2k}U_3^{-1})(U_3A_1^{-1}) \in L \textrm{ and}\\
B_1B_2^{-1}&=&T^3(B_{2k}A_1^{-1})T^{-3} \in L
\end{eqnarray*}
by the fact that $T^3(b_{2k},a_1)=(b_1,b_2)$. Since the elements $T$, $A_1A_2^{-1}$ and $B_1B_2^{-1}$ are in $L$, one can conclude that the generators $A_1$, $A_2$, $B_1, \ldots,B_{2k}$ and $C_1,\ldots,C_{2k-1}$ belong to $L$ by the proof of Theorem~\ref{t1}. Thus, we have  
\begin{eqnarray*}
U_3&=&(U_3B_{2k}^{-1})(B_{2k})\in L,\\
B_{2k+1}&=&(B_{2k+1}B_{2k}^{-1})(B_{2k})\in L
\end{eqnarray*}
and since $T^{-3}(u_3)=c_{2k}$, the element
\[
C_{2k}=T^{-3}U_3T^{3} \in L.
\]
It remains to prove that the subgroup $L$ contains the Dehn twist $E$. It is easy to obtained that 
\[
F_{g-2}=(F_{g-2}U_3^{-1})(U_3) \in L.
\]
Let $\varphi$ be the diffeomorphism
\[
T^{3}C_{2k}B_{2k+1}C_{2k-1}B_{2k}C_{2k}U_{g-4}^{-1}T^{-2},
\]
which is contained in the subgroup $L$. It can be verified that $\varphi$ takes the curve $f_{g-2}$ to $f_{1}$.
Therefore, we have 
\[
F_1=\varphi F_{g-2}\varphi^{-1}\in L.
\]
This finishes the proof since $E=A_1F_1A_1^{-1}$ is contained in $L$.
\end{proof}

We now ready to give the main theorem of this section.
\begin{theorem}\label{tcom}
The twist subgroup $\mathcal{T}_g$ is generated by
\begin{enumerate}
\item[(1)] two commutators if $g=4k\geq44$ or $g=4k+1\geq29$ and
\item[(2)] three commutators if $g=4k+2\geq30$ or $g=4k+3\geq43$.
\end{enumerate}
\end{theorem}
\begin{proof}
We will prove our results in four cases.\\

\underline{$g=4k\geq44$}: By Theorem~\ref{t42}, $\mathcal{T}_g$ is generated by the elements $T$ and $\Gamma_{10}C_2^{-1}F_{18}D_{33}^{-1}$, where the rotation $T=\rho_2\rho_1$ is depicted in Figure~\ref{R4K}. By Proposition~\ref{prop}, $T$ equals to a single commutator in $\mathcal{T}_g$. On the other hand, it is clear that there is a diffeomorphism $\phi$ that can be chosen as a product of Dehn twists such that it maps the curves $(\gamma_{10},f_{18})$ to the curves $(d_{33}, c_2)$, respectively. Hence
\begin{eqnarray*}
\Gamma_{10}C_2^{-1}F_{18}D_{33}^{-1}&=&\Gamma_{10}F_{18}(D_{33}C_2)^{-1}\\
&=&\Gamma_{10}F_{18}(\phi \Gamma_{10}F_{18}\phi^{-1})^{-1}\\
&=&(\Gamma_{10}F_{18}\phi)(\Gamma_{10}F_{18})^{-1}\phi^{-1}\\
&=&[\Gamma_{10}F_{18},\phi].
\end{eqnarray*}\\
\underline{$g=4k+1\geq29$}: Theorem~\ref{t29} implies that $\mathcal{T}_g$ is generated by the elements $T$ and $\Gamma_{10}C_2^{-1}F_{18}C_{12}^{-1}$. The rotation $T=\rho_2\rho_1$ in Figure~\ref{R4K+1} is a single commutator by Proposition~\ref{prop}. Also, it is clear that there is a diffeomorphism $\varphi$ contained in $\mathcal{T}_g$ so that it sends  $(\gamma_{10},f_{18})$ to $(c_{12}, c_2)$. By a similar argument as above, we get
\[
\Gamma_{10}C_2^{-1}F_{18}C_{12}^{-1}=[\Gamma_{10}F_{18},\varphi].
\]
This proves our sharpest result $(1)$.\\

\underline{$g=4k+2\geq30$}: In this case, $\mathcal{T}_g$ is generated by the elements $T$, $\Gamma_{10}C_2^{-1}F_{18}B_{2k}^{-1}$ and $C_{2k-1}D_{4k+1}^{-1}$ by Theorem~\ref{t4k+2}. The rotation $T=\rho_2\rho_1$ in Figure~\ref{R4K+3} is a single commutator by Proposition~\ref{prop}. Also, it is easy to verify that there exist diffeomorphisms $\psi_1$ and $\psi_2$ in $\mathcal{T}_g$ such that $\psi_1$ sends  $(\gamma_{10},f_{18})$ to $(b_{2k}, c_2)$ and $\psi_2$ sends $c_{2k-1}$ to $d_{4k+1}$. Using again a very similar argument as above, we get
\[
\Gamma_{10}C_2^{-1}F_{18}B_{2k}^{-1}=[\Gamma_{10}F_{18},\psi_1],
\]
and also
\begin{eqnarray*}
C_{2k-1}D_{4k+1}^{-1}&=&C_{2k-1}(\psi_{2}C_{2k-1}\psi_{2}^{-1})^{-1}\\
&=&[C_{2k-1},\psi_2].
\end{eqnarray*}\\
\underline{$g=4k+3\geq43$}: The twist subgroup $\mathcal{T}_g$ is generated by the elements $T$, $\Gamma_{10}C_2^{-1}F_{18}U_{33}^{-1}$ and $B_{2k+1}A_{1}^{-1}$ 
by Theorem~\ref{t4k+3}. Once again $T=\rho_2\rho_1$ in Figure~\ref{R4K+3} is a commutator by Proposition~\ref{prop} and so are the elements 
$\Gamma_{10}C_2^{-1}F_{18}U_{33}^{-1}$ and $B_{2k+1}A_{1}^{-1}$ by the same argument above.
\end{proof}

If one wants to decrease $g$, then again the number of commutator generators increases. In the remaining parts, we try to find minimal number of commutator generators for $g\geq7$.

\begin{theorem}\label{tcom1}
The twist subgroup $\mathcal{T}_g$ is generated by
\begin{enumerate}
\item[(1)] three commutators if $g=4k\geq8$ or $g=4k+1\geq9$ and
\item[(2)] four commutators if $g=4k+2\geq10$ or $g=4k+3\geq7$.
\end{enumerate}
\end{theorem}
\begin{proof}
The proof is very similar to the proof of Theorem~\ref{tcom}.\\

\underline{$g=4k\geq 8$}: By Theorem~\ref{t8even}, $\mathcal{T}_g$ is generated by the elements $T$, $D_{g-1}A_{2}^{-1}$ and $F_1B_{2}^{-1}$, where the rotation $T=\rho_2\rho_1$ is depicted in Figure~\ref{R4K}. By the argument similar to that in the proof of Theorem~\ref{tcom}, each generator can be expressed as a single commutator.

\underline{$g=4k+1\geq9$}: Theorem~\ref{t9odd} implies that $\mathcal{T}_g$ is generated by the elements $T$ in Figure~\ref{R4K+1}, $A_1A_2^{-1}$ and $F_1B_2^{-1}$. We now apply the same argument above, to express each generator as a commutator in $\mathcal{T}_g$,
which proves the result $(1)$.

\underline{$g=4k+2\geq10$}: In this case, using the generating set for $\mathcal{T}_g$ in Theorem~\ref{t4k+2.10} and the similar argument in the proof of above theorem, one can prove that the rotation $T=\rho_2\rho_1$ in Figure~\ref{R4K+3} and each generator can be written as a single commutator. 

\underline{$g=4k+3\geq7$}: In this last case, we use the generating set given in Theorem~\ref{t4k+3.7}. Each generator can be written as a single commutator using similar argument above, where the first generator $T$ is shown in Figure~\ref{R4K+3}.
\end{proof}

\end{document}